\documentclass[11pt,reqno]{amsart}

\usepackage{hyperref}
\usepackage{color}
\usepackage{amssymb}

\usepackage{graphicx, comment}
\usepackage{enumerate}
\usepackage{ytableau}
\usepackage{amsmath}
\usepackage{mathtools}
\usepackage{graphicx}
\usepackage{tikz}

\tikzstyle{sq}=[square, minimum width=1cm, text centered, draw=black]
\tikzstyle{filled}=[square, minimum width=1cm, text centered, draw=black, fill=green]

\usetikzlibrary{decorations.pathreplacing}

\usetikzlibrary{matrix,positioning}

\newtheorem{Theorem}{Theorem}
\newtheorem{Proposition}[Theorem]{Proposition}

\newtheorem{Lemma}[Theorem]{Lemma}
\newtheorem{Definition}{Definition}

\theoremstyle{remark}
\newtheorem*{Remark}{Remark}
\newtheorem*{Notation}{Notation}

\newtheorem*{rems}{Remarks} 
\newenvironment{Remarks}{\begin{rems}\normalfont}{\end{rems}}

\newtheorem{Example}{Example}

\numberwithin{equation}{section}

\allowdisplaybreaks 

\newcommand{\qrfac}[2]{{\left({#1}; q\right)_{#2}}} 

\newcommand{\elliptictheta}[1]{\theta\!\left({#1} ; p\right) }
\newcommand{\qbin}[2]{\begin{bmatrix} #1\\ #2 \end{bmatrix}_q}

\author[G.~Bhatnagar]{Gaurav Bhatnagar
}

\address{Ashoka University, Sonipat, Haryana 131029, India}
\email{bhatnagarg@gmail.com}

\author[A.~Kumari]{Archna Kumari}
\address{Department of Mathematics, Indian Institute of Technology, Delhi 110067, India.}
\email{arcyadav856@gmail.com}

\author[M.\ J.\ Schlosser]{Michael J.\ Schlosser}
\address{Fakult\"at f\"ur Mathematik\\
Universit\"at Wien\\
Oskar-Morgenstern-Platz~1\\
A-1090 Vienna, Austria}
\email{michael.schlosser@univie.ac.at}

\title{A weighted extension of Fibonacci numbers}
 
\subjclass{Primary 05A19 ; Secondary 11B39, 11B65, 33E05}

\keywords{Fibonacci numbers, weighted extensions, elliptic extensions}

\begin{document}

\begin{abstract}
We extend Fibonacci numbers with arbitrary weights and generalize a dozen Fibonacci identities. As a special case, we propose an elliptic extension which extends the $q$-Fibonacci polynomials appearing in Schur's work. The proofs of most of the identities are combinatorial, extending the proofs given by Benjamin and Quinn, and in the $q$ case, by Garrett.  Some identities are proved by telescoping. 
\end{abstract}

\maketitle

\section{Introduction}
The Fibonacci numbers are defined by the recurrence relation
$$F_{n+2} = F_{n+1}+F_n$$
with the initial values $F_0=0$ and $F_1=1$. 
According to Singh~\cite{Singh1985}, these numbers were obtained in ancient Indian texts (600--800 A.D.) in the context of explaining the rules of poetry composition in Sanskrit and Prakrit (two ancient Indian languages). Knowledge at the time was transferred orally; it was useful to present it in the form of verse, so that it was easier to memorise and replicate. To describe the
aforementioned rules, we need the following vocabulary. 
A {\em mora} is a unit of syllables (plural: {\em morae}). The basic units of Sanskrit prosody is a letter having one mora (called {\em laghu}) and with two morae (called {\em guru}). A type of metre (or line) in Sanskrit poetry allows sequences of {\em laghu} and {\em guru} letters in any order, but the metre should contain exactly $n$ morae. It is in counting all possible metres of this type that the so-called Fibonacci numbers arose. 

In more recent times, many Fibonacci identities have been proved in much the same way,  by Benjamin and Quinn~\cite{BQ2003} and  Benjamin, Quinn and Rouse~\cite{BQR2004}. These authors count tilings of a $1\times n$ board by squares and dominos.  The square corresponds to a {\em laghu} and a domino to a
{\em guru}.  A tiling of an $n$-board with squares and dominos corresponds to a metre with exactly $n$ morae made up of the corresponding letters. 

The objective of this paper is to obtain an extension of the Fibonacci numbers by weighted counting---with arbitrary weights---in much the same way as described above. We also suggest an elliptic weight, which results in elliptic extensions of Fibonacci numbers. These extensions have 
additional parameters $q$ (the {\em base}), $p$ (the {\em nome}) as well as two arbitrary parameters $a$ and $b$. In special cases, they reduce to $q$-Fibonacci polynomials that were considered by Schur in work related to the Rogers--Ramanujan identities. 

We work in the context of elliptic 
combinatorics. This field has been developing rapidly, due to efforts by Schlosser, Yoo and others~\cite{ BCK2020a, BCK2020b, MS2007, MS2020b, SY2015,SY2016b, SY2016a,  SY2017b, SY2017a,  SY2018,  SY2021}. Our elliptic extension makes a small change in the definition of elliptic Fibonacci numbers by Schlosser and Yoo~\cite{SY2017a}. This small change allows extensions of many well-known Fibonacci identities. 
Combinatorial proofs in this paper are modeled on
Garrett's unpublished work~\cite{Garrett2004}, which, in turn, are $q$-analogues of the proofs in Benjamin and Quinn's exposition. A very similar approach has been taken by Cigler~\cite{Cigler2003}.

Using weighted counting, we extend many well-known Fibonacci identities. In particular, we give a weighted generalization (with a combinatorial proof) of 
\begin{equation}\label{vajda}
F_{n+i}F_{n+j}-F_{n}F_{n+i+j}=(-1)^{n}F_{i}F_{j}.
\end{equation}
This was proposed as a problem by Everman,  Danese and Venkannayah~\cite{EDV1960} and appears in Vajda~\cite[p.~177, Equation (20a)]{Vajda1989}. It extends some
classical Fibonacci identities, due to Cassini (1680):
$$F_{n}^2 -F_{n-1}F_{n+1} = (-1)^{n-1};$$ 
and Catalan (1879): 
$$F_{n}^{2}-F_{n-i}F_{n+i}=(-1)^{n-i}F_{i}^{2}.$$
A special case of our result was considered by Cigler~\cite{Cigler2003}. 

This paper is organized as follows. In Sections~\ref{sec:def1} and \ref{sec:def2} we give the analytic and combinatorial definitions of the weighted Fibonacci numbers, as well as the proposed elliptic weights. Section~\ref{sec:weigh} contains weighted analogues of several Fibonacci identities. 
Finally, in Section~\ref{sec:telescoping}, we obtain some Fibonacci identities by using the telescoping methods given in \cite{GB2011, GB2016}. 

\section{Weighted Fibonacci numbers: definitions}\label{sec:def1}

\begin{Definition}[First definition of weighted Fibonacci numbers]\label{def1} Let $(w^f_n)_{n\ge 0}$ be a sequence of 
indeterminates called weights.
We define the weighted Fibonacci number $f_{n}$ as $f_{0}=0$, $f_{1}=1,$ and for $n\ge 0$
\begin{equation}\label{EFibdef1}
f_{n+2} = f_{n+1} + w^f_n f_n.
\end{equation}
\end{Definition}
We have suppressed the dependence on the sequence of weights $(w_n^f)$ in our notation. To define the elliptic weights, we need some notation.


The {\bf modified Jacobi theta function} with nome $p$ is defined as
\begin{equation*} \elliptictheta{a} := \prod\limits_{j=0}^{\infty} (1-ap^j) (1-p^{j+1}/a)
\end{equation*}
where $a\neq 0$ and $|p|<1$. 
When the nome $p=0$, the modified theta function $\elliptictheta{a}$ reduces to $(1-a)$.
We use the short-hand notation
\begin{equation*}
\elliptictheta{a_1, a_2, \dots, a_r} := \elliptictheta{a_1} \elliptictheta{a_2}\cdots  \elliptictheta{a_r}.
\end{equation*}

Two important properties of the modified theta function are
\cite[Equation~(11.2.42)]{GR90}
\begin{subequations}
\begin{gather}\label{GR11.2.42}
\elliptictheta{a} =\elliptictheta{p/a} =-a\elliptictheta{1/a},
\end{gather}
and \cite[p.~451, Example~5]{WW1996}
\begin{gather}\label{addf}
\elliptictheta{xy,x/y,uv,u/v}-\elliptictheta{xv,x/v,uy,u/y}=\frac uy\,\elliptictheta{yv,y/v,xu,x/u}.
\end{gather}
\end{subequations}

Let $g(x)$ be an elliptic function, that is, a~doubly periodic meromorphic function of the complex variable~$x$.
Without loss of generality, by the theory of theta functions, we may assume that
\begin{gather*}
g(x)= \frac{\elliptictheta{ a_1q^x,a_2q^x,\dots,a_{r}q^x}} {\elliptictheta{b_1q^x,b_2q^x,\dots,b_rq^x}} c,
\end{gather*}
where $c$ is a constant, and the {\em elliptic balancing condition},
\begin{gather*}
a_1a_2\cdots a_{r}=b_1b_2\cdots b_r,
\end{gather*}
holds. If we write $q=e^{2\pi i\sigma}$, $p=e^{2\pi i\tau}$, with complex $\sigma$, $\tau$, then $g(x)$ is indeed doubly periodic in~$x$ with periods $\sigma^{-1}$ and $\tau\sigma^{-1}$. 

These definitions motivate the form of the elliptic weights defined in this paper. 

\begin{Definition}[Elliptic Fibonacci weights]
For the nome $p$, {\em base} $q$, two independent variables $a$ and $b$, and $n=0, 1, 2, \dots$, define the weight function
\begin{equation}\label{elliptic-fib-weight}
w^{f}_{n}(a,b; q,p) :=\frac{\theta(aq,aq^{2},bq^{1-2n}, aq/b, a/b;p)}{\theta(aq^{1-n},aq^{2-n},bq,aq^{1+n}/b,aq^n/b;p)}q^{n}.
\end{equation}
\end{Definition}
\begin{Notation}
We use the notation  $f_n(a,b; q,p)$ to denote the corresponding elliptic Fibonacci numbers, where we take 
$w^f_n :=w^f_n(a,b; q,p).$ In what follows, we use $(a,b;q,p)$ in our notation when referring to formulas for elliptic Fibonacci numbers. Otherwise, the  reference is to arbitrary (unspecialized) weights. 
\end{Notation}
\begin{Remarks}\
\begin{enumerate}
\item The expression \eqref{elliptic-fib-weight} is elliptic in the variables $\log_q a$, $\log_q b$ and $n$. Instead of $\log_q a$, we may also replace $a$ by $q^a$, then it is elliptic in the variable $a$.  The analogous statement holds for the variable $b$. 
\item We have modified the definition of the elliptic Fibonacci numbers given by
Schlosser and Yoo \cite[Definition 3.6]{SY2018}. Their $S_n(a,b;q,p)$ is related to our definition as follows
\begin{equation}
S_n(aq^{-n},bq^{-2n};q,p) = f_n(a,b;q,p).
\end{equation}
\item In \cite{SY2017a}, the authors defined the elliptic numbers $[n]_{a,b;q,p}$ using the
 weight function:
\begin{equation}\label{elliptic-weights}
W_{a,b;q,p}(k)=
\frac{\theta(aq^{2k+1},bq,bq^2,aq^{-1}/b,a/b;p)}
{\theta(aq,bq^{k+1},bq^{k+2},aq^{k-1}/b, aq^k /b;p)}q^k.
\end{equation}
The elliptic numbers $[n]_{a,b;q,p}$  satisfy the recurrence
\begin{equation}\label{elliptic-analogue}
[n]_{a,b; q, p}+W_{a,b; q,p}(n) [m-n]_{aq^{2n},bq^n; q, p} = [m]_{a,b; q, p}
\end{equation}
and are defined explicitly for complex $n$ as follows:
$$
[n]_{a,b}=
\frac{\theta(q^n,aq^n,bq^2,a/b;p)}
{\theta(q,aq,bq^{n+1},aq^{n-1}/b;p)}.
$$
Note that the weight function for the elliptic Fibonacci numbers is closely connected to \eqref{elliptic-weights}
by the relation:
$$w^f_n(a,b; q,p) = W_{b,a;q,p}(-n).$$

\end{enumerate}

\end{Remarks}

There are also $m$-versions of Fibonacci and $q$-Fibonacci polynomials in the literature. 
In the $q$-case, these arose in work of Schur and appeared, for example, in the $m$-versions of the Rogers--Ramanujan identities, given by Garrett, Ismail and Stanton~\cite{GIS1999}. These authors' work motivates the following definition. 
\begin{Definition}[Weighted $m$-shifted Fibonacci numbers]\label{def-m-fib}
Let $w_n^f$ be as in Definition~\ref{def1}. The weighted $m$-shifted Fibonacci numbers are given by
$f^{(m)}_{0}=0$, $f^{(m)}_{1}=1,$ and for $n\ge 0$
\begin{equation}\label{EFibdef2}
f^{(m)}_{n+2} = f^{(m)}_{n+1} + w^f_{n+m} f^{(m)}_n.
\end{equation}
\end{Definition}
\begin{Notation}
We use the notation  $f_n^{(m)}(a,b; q,p)$ to denote the corresponding $m$-shifted elliptic Fibonacci numbers, when the weight sequence is the one given in \eqref{elliptic-fib-weight}.

\end{Notation}

\subsection*{Special cases} We note some special cases of the elliptic weight function which define extensions of Fibonacci numbers with additional parameters. 
\begin{enumerate}
\item When $p=0$, we obtain an $a,b;q$-analogue of Fibonacci numbers. The weight function reduces to:
$$w^{f}_{n}(a,b; q) :=\frac{(1-aq)(1-aq^{2})(1-bq^{1-2n})(1- aq/b)(1- a/b)}
{(1-aq^{1-n})(1-aq^{2-n})(1-bq)(1-aq^{1+n}/b)(1-aq^n/b)}q^{n}.
$$
\item When $b\to 0$ or $\infty$, we obtain an $a;q$-analogue of Fibonacci numbers. The weight is
\begin{equation}\label{aqweight}
w^{f}_{n}(a; q) :=\frac{(1-aq)(1-aq^{2})}
{(1-aq^{1-n})(1-aq^{2-n})}q^{-n}.
\end{equation}
\item A $b;q$-analogue of the Fibonacci numbers is obtained by taking $a\to 0 $ or $\infty$ in $w^f_n(a,b;q)$. The weight function is 
$$w^{f}_{n}(b; q) :=\frac{(1-bq^{1-2n})}
{(1-bq)}q^{n}.
$$
\item Finally, when $b\to 0$ and $a\to 0$, we obtain the $q$-Fibonacci numbers determined by the weight function $w_n(q)=q^n$. 
\end{enumerate}

\section{Fibonacci numbers from weighted counting}\label{sec:def2}
In this section, we give an alternative combinatorial definition of the weighted Fibonacci numbers.

Consider a $1\times n$ board (henceforth, $n$-board) of square boxes. On this $n$-board one can place a domino ($1\times 2$ tile) or a square ($1\times 1$ tile). If a domino is placed at a location covering the boxes at $i$ and $i+1$, we give it a weight $w_i^f$. (The weight of a square tile is $1$.) Given a tiling $T$ of an $n$-board with dominos and squares, the weight of $T$ is defined as 
$$W^f(T) := \prod_{i} w_i^f ,$$
where the product is over all positions $(i,i+1)$ containing a domino. An empty product is defined to be $1$. 

\begin{Definition}[Second definition of the weighted Fibonacci numbers]\label{def2}
Let $n$ be a non-negative integer. We take $f_0:=0$ and for $n\ge 0$,
$$ f_{n+1} :=\sum_{T_{n}} W^f(T_n)$$ 
where the sum runs over all tilings $T_{n}$ of an $n$-board. We say that $f_{n+1}$ is the {\em \bf total weight} of tilings 
of an $n$-board.
\end{Definition} 

\begin{Remark} 
The weight of a pair $(S, T)$ of tilings is the product of their weights. This corresponds to the weighted counting of a pair of sets. One can view the pair as a concatenation of tilings.
\end{Remark}
%

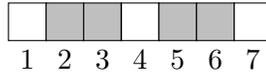
\begin{Example}
\

\begin{center}
  \begin{tikzpicture}

\draw[ draw=black] (0,0) rectangle   (0.5,0.5);
\node[rectangle] at (0.25,-0.25) {1};
\filldraw[fill=gray!50, draw=black] (0.5,0) rectangle   (1,0.5);
\node[rectangle] at (0.75,-0.25) {2};
\filldraw[fill=gray!50, draw=black] (1,0) rectangle   (1.5,0.5);
\node[rectangle] at (1.25,-0.25) {3};
\draw[ draw=black] (1.5,0) rectangle   (2,0.5);
\node[rectangle] at (1.75,-0.25) {4};
\filldraw[fill=gray!50, draw=black] (2,0) rectangle   (2.5,0.5);
\node[rectangle] at (2.25,-0.25) {5};
\filldraw[fill=gray!50, draw=black] (2.5,0) rectangle   (3,0.5);
\node[rectangle] at (2.75,-0.25) {6};
\draw[ draw=black] (3,0) rectangle   (3.5,0.5);
\node[rectangle] at (3.25,-0.25) {7};

\end{tikzpicture}
\end{center}
In this tiling  the two gray boxes represent a domino and a white box represents a square. It has weight $w_{2}^fw_{5}^f$ as there is one domino in position $(2,3)$ and one in position $(5,6).$
\end{Example}

It is easy to see that the two definitions are equivalent. Let $f_{n}$ be as in Definition~\ref{def2}.  Clearly, $f_0 =0$ and $f_1 =1$. Let $n>1$. Then consider any tiling of an $(n+1)$-board. It's last square is tiled by either a square or a domino. If it is a square, it contributes $W^f(T_n)$ to $f_{n+2}$, where $T_n$ is a tiling of an $n$-board obtained by deleting the last square. Otherwise, it contributes $w_n^f W^f (T^\prime_{n-1})$ where $T^\prime_{n-1}$ is a tiling of the $n-1$ board, obtained by deleting the last 2 squares. Summing over all tilings of an $(n+1)$-board, we see that $f_{n}$ satisfies the recurrence
\eqref{EFibdef1}, and also satisfies Definition~\ref{def1}.

%

The $m$-shifted weighted Fibonacci numbers can be defined similarly. In an $(m+n)$-board, no domino is allowed to be placed in the first $m$ positions; the generating function so obtained is
defined to be
$f_{n+1}^{(m)} $. The number of positions where dominos can be placed is still $n$.
Clearly $f_{n+1}^{(0)} =f_{n+1} $. 

\begin{Example}
Consider the 7-board where no dominos are allowed in the first three positions. 
The following tiling is allowed, and has weight $w_4^fw_{6}^f$.
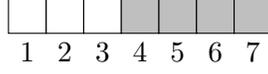
\begin{center}
  \begin{tikzpicture}
\draw[ draw=black] (0,0) rectangle   (0.5,0.5);
\node[rectangle] at (0.25,-0.25) {1};
\draw[draw=black] (0.5,0) rectangle   (1,0.5);
\node[rectangle] at (0.75,-0.25) {2};
\draw[ draw=black] (1,0) rectangle   (1.5,0.5);
\node[rectangle] at (1.25,-0.25) {3};
\filldraw[fill=gray!50, draw=black] (1.5,0) rectangle   (2,0.5);
\node[rectangle] at (1.75,-0.25) {4};
\filldraw[fill=gray!50,draw=black] (2,0) rectangle   (2.5,0.5);
\node[rectangle] at (2.25,-0.25) {5};
\filldraw[fill=gray!50, draw=black] (2.5,0) rectangle   (3,0.5);
\node[rectangle] at (2.75,-0.25) {6};
\filldraw[fill=gray!50, draw=black] (3,0) rectangle   (3.5,0.5);
\node[rectangle] at (3.25,-0.25) {7};

\end{tikzpicture}
\end{center}
Consider all possible tilings of a 7-board where no domino is allowed in positions 1, 2 and 3. We obtain 
$$f_{5}^{(3)}=1+w_{4}^f+w_{5}^f+w_{6}^f+w_{4}^fw_{6}^f.$$

\end{Example}

A {\bf fault} of a tiling is defined as a position $m$ such that the tiling does not include a domino in position $(m,m+1).$ Occasionally we will need to split a tiling at a fault and count the possible tilings to the right and to the left of the fault. While tilings to the left of a fault are counted by weighted Fibonacci numbers, tilings to the right have shifted weights, and are thus counted by shifted Fibonacci numbers. 

%
%

Next, we define an expression analogous to the binomial coefficient. 
\begin{Proposition} Let $g^n_k $ be the weighted sum of tilings of $n$ tiles with exactly $k$ dominos. Then  $g^n_k $ is determined by the following:
\begin{subequations}
\begin{align}
g_k^n  &= w_{n+k-1}^f g_{k-1}^{n-1} + g_k^{n-1} ; \label{Newbino}\\
\intertext{where the initial conditions are as follows:}
g_0^n  &= 1 \text{ for } n\ge 0; \\
g_k^n  &= 0 \text{ for } k < 0 \text{ and } k>n.
\end{align}
\end{subequations}
\end{Proposition}

\begin{Remark}
As before, we denote by $g_k^n(a,b; q,p)$ when the weights are specified by \eqref{elliptic-fib-weight}.
\end{Remark}
\begin{proof}
The initial conditions are easy to check. Further, it is easy to check that these determine $g_k^n$ for non-negative integers $n$ and $k$. 

The total number of squares in a tiling with exactly $k$ dominos and $n-k$ squares is $n+k$. The last square can either be covered by a square tile or there is a domino at $(n+k-1,n+k)$. If the last tile is a square, then the sum of weights of the tilings is $g_{k}^{n-1} $. Otherwise, the domino has weight $w_{n+k-1}^f$ and tilings after deleting this domino have weight 
$g_{k-1}^{n-1} $. The total weight when there is a domino at $(n+k-1,n+k)$ is thus
$w_{n+k-1}^f g_{k-1}^{n-1} $. This shows \eqref{Newbino}.
\end{proof}
\begin{Remarks} \ 
\begin{enumerate}
\item Note that, for $n\ge 1$,
$$g_1^n  = \sum_{i=1}^{n} w_{i}^f$$
and 
$$g_n^n  = w_1^fw_3^f\cdots w_{2n-1}^f.$$
\item In the $q$-case, $g^n_k $ reduces to $q^{k^2}\qbin{n}{k}$. 
\item If we define
$$\begin{bmatrix} n\\ k\end{bmatrix}_w := \frac{g^n_k}
{w_1^f w_3^f\cdots w_{2k-1}^f},$$
then $\begin{bmatrix} n\\ k\end{bmatrix}_w$ satisfies
\begin{equation}\label{wbin}
\begin{bmatrix} n\\ k\end{bmatrix}_w
= \frac{w_{n+k-1}^f}{w_{2k-1}^f}
\begin{bmatrix} n-1\\ k-1\end{bmatrix}_w 
+\begin{bmatrix} n-1\\ k\end{bmatrix}_w .
\end{equation}
\item In addition, in the $a;q$-case (that is, when $p=0$ and $b\to 0$), the weight is given by \eqref{aqweight}. In this case,  
$\begin{bmatrix} n\\ k\end{bmatrix}_w$ factorizes and 
reduces to:
\begin{equation}\label{qbin-aq}
\begin{bmatrix} n\\ k\end{bmatrix}_{a;q} = \frac{\qrfac{q^{k-1}/a}{k}}{\qrfac{q^{n-1}/a}{k}}\begin{bmatrix} n\\ k\end{bmatrix}_q.
\end{equation}
It is easy to verify this formula satisfies the three term recurrence relation \eqref{wbin} in the $a;q$ case. 
The factorization in the $a;q$ case is related to one obtained for $S_n(a;q)$ in \cite{SY2018}.
\end{enumerate}
\end{Remarks}

\section{Proofs that really weigh}\label{sec:weigh}
We now prove analogues of several Fibonacci identities, using counting---or rather weighing---tilings of boards. 
We would like to remind the reader that in any of the theorems of this paper, the corresponding results for $q$-Fibonacci numbers are obtained by formally replacing $w_k^f$ by $q^k$. Replacing all the $w_k^f$ by $1$ reduces any result to one for the Fibonacci numbers. 

\begin{Theorem} Let $n$ be a positive integer. Then
\begin{equation*}
\sum_{k=1}^{n} w^{f}_{k} f_{k} 
=f_{n+2} -1. 
\end{equation*}
\end{Theorem}
\begin{proof}
We weigh the tilings of an $(n+1)$-board with at least one domino.

The total number of tilings of an $(n+1)$-board is $f_{n+2} $. To obtain the right-hand side, we remove the tiling 
consisting of all square tiles from these.  

To obtain the left-hand side, suppose the position of the last domino on the board is at the $k$th tile, for some $1\le k \le n$. For each such $k$, the tilings of a $(k-1)$-board have weight $f_{k} $. Since the $k$th tile is a domino, the total weight of such tilings is $w_{k}^f f_{k} $.  The sum over all $k$ gives the left-hand side.
\end{proof}

\begin{Theorem} \label{Fib-71} Let $n$ and $m$ be positive integers. Then 
\begin{equation*}
f_{m+n+1} 
=f_{m+1}  f^{(m)}_{n+1}  
+w^{f}_{m}f_{m}  f^{(m+1)}_{n} .
\end{equation*}
\end{Theorem}
\begin{proof}
We weigh tilings of an $(m+n)$-board in two ways. 

The total weight of tilings of an $(m+n)$-board is $f_{m+n+1} $, which is the left-hand side.

To get the right-hand side, consider the $m$th position. Either there is a fault at position $m$, or there is a domino at $(m,m+1)$. 

If there is a fault at $m$, we can split the board after position $m$ to obtain a pair of boards of length $m$ and $n$. The weight of tilings of the first board is $f_{m+1} $ and the weight of the second board
is the $m$-shifted Fibonacci number $f^{(m)}_{n+1}(a,b; q,p)$. This gives the total weight as
$$f_{m+1}  f^{(m)}_{n+1} .$$ 

In case there is a domino at the position $(m,m+1)$, this domino contributes the weight $w^{f}_{m}$. On breaking the board after position $m-1$ and $m+1$, we see
the weight of the pair of tilings of the remaining $1\times m-1$ and $1\times n-1$ boards is
$f_{m}  f^{(m+1)}_{n} .$
The total weight is the product is $$w^{f}_{m}f_{m}  f^{(m+1)}_{n} .$$
Adding the two gives the right-hand side. 
\end{proof}

Next, we give an extension of \eqref{vajda} which was highlighted in the introduction. 

\begin{Theorem}\label{Fib-81} Let $n$ be a positive integer. Then
\begin{equation*}
f_{n+i+1} f^{(i+1)}_{n+j+1} 
=
f_{n+i+j+2}  f^{(i+1)}_{n}  
+(-1)^nw^{f}_{i+1}w^{f}_{i+2} \cdots w^{f}_{i+n}
f_{i+1}  f_{j+1}^{(n+i+1)}  .
\end{equation*}
\end{Theorem}
\begin{Remark} When $i=0$, in the $q$-case, an equivalent form of this identity has been given a combinatorial proof by Cigler~\cite[Equation (5.6)]{Cigler2003}.
\end{Remark}
\begin{proof}
Consider two boards. One of the boards is of length $i+n$; the other of length $n+j$, placed one below the other, with the lower board shifted $i+1$ places, as shown. We count the number of weighted tilings of these boards. 

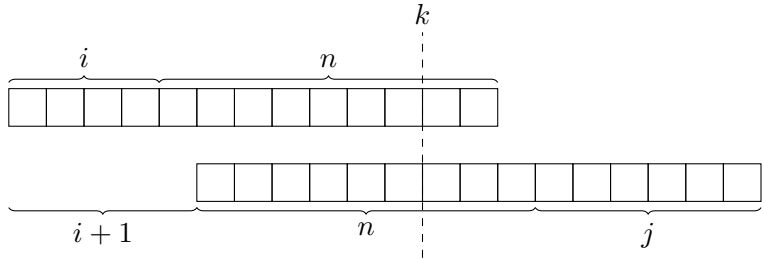
\begin{align*}
&
\begin{tikzpicture}
\draw[decorate, decoration={brace}, yshift=2ex] (-2,0.25) -- node[above=0.4ex] {$i$} (0,0.25);
\draw[ draw=black] (-2,0) rectangle   (-1.5,0.5);
\draw[draw=black] (-1.5,0) rectangle   (-1,0.5);
\draw[ draw=black] (-1,0) rectangle   (-.5,0.5);
\draw[ draw=black] (-.5,0) rectangle   (0,0.5);
\draw[decorate, decoration={brace}, yshift=2ex] (0,0.25) -- node[above=0.4ex] {$n$} (4.5,0.25);
\draw[ draw=black] (0,0) rectangle   (0.5,0.5);
\draw[draw=black] (0.5,0) rectangle   (1,0.5);
\draw[ draw=black] (1,0) rectangle   (1.5,0.5);
\draw[ draw=black] (1.5,0) rectangle   (2,0.5);
\draw[ draw=black] (2,0) rectangle   (2.5,0.5);
\draw[ draw=black] (2.5,0) rectangle   (3,0.5);
\draw[ draw=black] (3,0) rectangle   (3.5,0.5);
\draw[ draw=black] (3.5,0) rectangle   (4,0.5);
\draw[ draw=black] (4,0) rectangle   (4.5,0.5);
\draw [dashed] (3.5, -1.75) -- (3.5, 1.25);
\node [label] at (3.5, 1.5) {$k$};
\draw[decorate, decoration={brace, mirror}, yshift=-2ex] (-2,-0.75) -- node[below=0.4ex] {$i+1$} 
(.5,-0.75);
\draw[ draw=black] (.5,-1) rectangle   (1,-0.5);
\draw[ draw=black] (1,-1) rectangle   (1.5,-0.5);
\draw[draw=black] (1.5,-1) rectangle   (2,-0.5);
\draw[ draw=black] (2,-1) rectangle   (2.5,-0.5);
\draw[draw=black] (2.5,-1) rectangle   (3,-0.5);
\draw[ draw=black] (3,-1) rectangle   (3.5,-0.5);
\draw[ draw=black] (3.5,-1) rectangle   (4,-0.5);
\draw[ draw=black] (4,-1) rectangle   (4.5,-0.5);
\draw[ draw=black] (4.5,-1) rectangle   (5,-0.5);
\draw[decorate, decoration={brace, mirror}, yshift=-2ex] (0.5,-0.75) -- node[below=0.4ex] {$n$} (5,-0.75);
\draw[decorate, decoration={brace, mirror}, yshift=-2ex] (5,-0.75) -- node[below=0.4ex] {$j$} (8,-0.75);
\draw[ draw=black] (5,-1) rectangle   (5.5,-0.5);
\draw[draw=black] (5.5,-1) rectangle   (6,-0.5);
\draw[ draw=black] (6,-1) rectangle   (6.5,-0.5);
\draw[ draw=black] (6.5,-1) rectangle   (7,-0.5);
\draw[ draw=black] (7,-1) rectangle   (7.5,-0.5);
\draw[ draw=black] (7.5,-1) rectangle   (8,-0.5);
\end{tikzpicture}
\\
& 
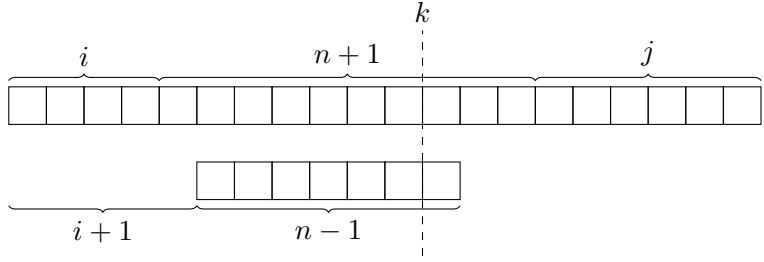
\begin{tikzpicture}
\draw[decorate, decoration={brace}, yshift=2ex] (-2,0.25) -- node[above=0.4ex] {$i$} (0,0.25);
\draw[ draw=black] (-2,0) rectangle   (-1.5,0.5);
\draw[draw=black] (-1.5,0) rectangle   (-1,0.5);
\draw[ draw=black] (-1,0) rectangle   (-.5,0.5);
\draw[ draw=black] (-.5,0) rectangle   (0,0.5);
\draw[decorate, decoration={brace}, yshift=2ex] (0,0.25) -- node[above=0.4ex] {$n+1$} (5,0.25);
\draw[ draw=black] (0,0) rectangle   (0.5,0.5);
\draw[draw=black] (0.5,0) rectangle   (1,0.5);
\draw[ draw=black] (1,0) rectangle   (1.5,0.5);
\draw[ draw=black] (1.5,0) rectangle   (2,0.5);
\draw[ draw=black] (2,0) rectangle   (2.5,0.5);
\draw[ draw=black] (2.5,0) rectangle   (3,0.5);
\draw[ draw=black] (3,0) rectangle   (3.5,0.5);
\draw[ draw=black] (3.5,0) rectangle   (4,0.5);
\draw[ draw=black] (4,0) rectangle   (4.5,0.5);
\draw[ draw=black] (4.5,0) rectangle   (5,0.5);
\draw [dashed] (3.5, -1.75) -- (3.5, 1.25);
\node [label] at (3.5, 1.5) {$k$};
\draw[decorate, decoration={brace}, yshift=2ex] (5,0.25) -- node[above=0.4ex] {$j$} (8,0.25);
\draw[ draw=black] (5,0) rectangle   (5.5,0.5);
\draw[draw=black] (5.5,0) rectangle   (6,0.5);
\draw[ draw=black] (6,0) rectangle   (6.5,0.5);
\draw[ draw=black] (6.5,0) rectangle   (7,0.5);
\draw[ draw=black] (7,0) rectangle   (7.5,0.5);
\draw[ draw=black] (7.5,0) rectangle   (8,0.5);
\draw[decorate, decoration={brace, mirror}, yshift=-2ex] (-2,-0.75) -- node[below=0.4ex] {$i+1$} 
(.5,-0.75);
\draw[ draw=black] (.5,-1) rectangle   (1,-0.5);
\draw[ draw=black] (1,-1) rectangle   (1.5,-0.5);
\draw[draw=black] (1.5,-1) rectangle   (2,-0.5);
\draw[ draw=black] (2,-1) rectangle   (2.5,-0.5);
\draw[draw=black] (2.5,-1) rectangle   (3,-0.5);
\draw[ draw=black] (3,-1) rectangle   (3.5,-0.5);
\draw[ draw=black] (3.5,-1) rectangle   (4,-0.5);
\draw[decorate, decoration={brace, mirror}, yshift=-2ex] (0.5,-0.75) -- node[below=0.4ex] {$n-1$} (4,-0.75);
\end{tikzpicture}
\end{align*}

The number of tilings of the first $i+n$-board is $f_{n+i+1} $. Visualize it as an $i$-board
concatenated with an $n$-board. 
 The second board is shifted by $i+1$, so with weight $f^{(i+1)}_{n+j+1} $. Thus the total weight of tilings of such  pairs of boards is 
$$f_{n+i+1} f^{(i+1)}_{n+j+1} .$$
This gives the left-hand side. 

For the right-hand side, consider the last common fault of both the boards. The fault is only possible
in the parts of the board which are $n$-boards. 
Let it be at the $k$th cell. Break both the tilings after this fault and swap the tails to get a pair consisting of an $i+(n+1)+j$-tiling and an $(n-1)+j$-tiling (shifted by $i+1$). Their weight is $f_{n+i+j+2} f^{(i+1)}_{n} $. 

There is a one-to-one correspondence between all pairs of such tilings that have a fault. In case the two tilings do not have fault, there is an error term, which we now show is 
$$ (-1)^n w^{f}_{i+1}w^{f}_{i+2} \cdots w^{f}_{i+n} f_{i+1}  f_{j+1}^{(n+i+1)}  .$$

The tilings consisting of the $i$-board on top and $j$-board at the bottom have no role to play in the fault. 
The $n$-tilings do not have common fault if both the tilings consist of only dominos. If $n$ is odd, then $(n+1)$, $(n-1)$, both are even and we can have an all domino tiling. This contributes a weight $w^{f}_{i+1}w^{f}_{i+2} \cdots w^{f}_{i+n}$. This is multiplied by the weight of tilings in the $i$ cells
(with contribution $f_{i+1} $ and the $j$ cells (contribution: $f_{j+1}^{(n+i+1))} $). 
But since the two $n$-boards on the left-hand side cannot have an all-domino tiling, we get
\begin{equation*} 
f_{n+i+1} f^{(i+1)}_{n+j+1}  =
 f_{n+i+j+2} f^{(i+1)}_{n}  
 -w^{f}_{i+1}w^{f}_{i+2} \cdots w^{f}_{i+n}f_{i+1} f_{j+1}^{(n+i+1)} .
 \end{equation*}
If $n$ is even, then $(n+1)$  and $(n-1)$ are both odd and cannot have an all domino tiling on the right-hand side. But on the left-hand side, the two $n$-boards can have all domino tilings, contributing the weight 
$$ w^{f}_{i+1}w^{f}_{i+2} \cdots w^{f}_{i+n} f_{i+1}  f_{j+1}^{(n+i+1)}  .$$
Therefore, when $n$ is even,  
\begin{equation*} 
f_{n+i+1} f^{(i+1)}_{n+j+1}  =
 f_{n+i+j+2} f^{(i+1)}_{n}  
 +w^{f}_{i+1}w^{f}_{i+2} \cdots w^{f}_{i+n}f_{i+1} f_{j+1}^{(n+i+1)} .
 \end{equation*}
 This completes the proof.
\end{proof}

\begin{Theorem}\label{Fib-21} Let $n$ be a positive integer. Then
\begin{equation*}
\sum_{k=1}^n  \frac{w^{f}_{1}w^{f}_{3} \dots w^{f}_{2n-1}}{w^{f}_{1}w^{f}_{3} \dots w^{f}_{2k-1}}f_{2k} 
=f_{2n+1} -w^{f}_{1}w^{f}_{3} \dots w^{f}_{2n-1}. 
\end{equation*}
\end{Theorem}
\begin{proof}
We weigh the tilings of a $2n$-board with at least one square. 

The tilings of a $2n$-board have weight $f_{2n+1} $. The tiling of $2n$-board with no square is tiling with all dominos, with weight $w^{f}_{1}w^{f}_{3} \dots w^{f}_{2n-1}$. On subtracting this, we get the right-hand side.

To obtain the left-hand side, consider the last square of the tiling of a $2n$-board. Suppose it is at the $i$th tile for $0\le i \le 2n$. Since it is the last square, after this all tiles are dominos. This implies $i$ has to be an even number, say $i$ is $2k$. On the left side of this square, the tilings of $(2k-1)$-board have weight $f_{2k} $. The dominos on the right side of $i$th tile contribute 
$$w^{f}_{2k+1}w^{f}_{2k+3} \dots w^{f}_{2n-1} = \frac{w^{f}_{1}w^{f}_{3} \dots w^{f}_{2n-1}}{w^{f}_{1}w^{f}_{3} \dots w^{f}_{2k-1}}.$$ 
Summing over $i$ gives us the left-hand side.
\end{proof}

\begin{Theorem}\label{Fib-31} Let $n$ be a positive integer. Then
\begin{equation*}
\sum_{k=0}^n  \frac{w^{f}_{2}w^{f}_{4} \dots w^{f}_{2n}}{w^{f}_{2}w^{f}_{4} \dots w^{f}_{2k}}f_{2k+1} 
=f_{2n+2} . 
\end{equation*}
\end{Theorem}
\begin{proof}
We weigh the tilings of a $(2n+1)$-board. The details are similar to the proof of Theorem~\ref{Fib-21} and are omitted. 
%
%
\end{proof}

\begin{Theorem}\label{Fib-41} Let  $n$ be a positive integer. Then
\begin{equation*}
\sum_{k=0}^n  \frac{w^{f}_{1}w^{f}_{2} \dots w^{f}_{n}}{w^{f}_{1}w^{f}_{2} \dots w^{f}_{k}}f_{k+1} ^2
=f_{n+1}  f_{n+2} . 
\end{equation*}
\end{Theorem}
\begin{proof}
We weigh the tilings of a pair consisting of an $n$-board and an $(n+1)$-board.

The right-hand side is the total weight of the tilings of this pair of boards.

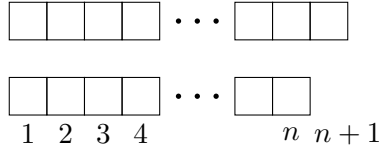
\begin{center}
  \begin{tikzpicture}

\draw[ draw=black] (0,0) rectangle   (0.5,0.5);
\draw[draw=black] (0.5,0) rectangle   (1,0.5);
\draw[ draw=black] (1,0) rectangle   (1.5,0.5);
\draw[ draw=black] (1.5,0) rectangle   (2,0.5);
\fill  (2.25, 0.25) circle[radius=1pt];
\fill  (2.50, 0.25) circle[radius=1pt];
\fill  (2.75, 0.25) circle[radius=1pt];
\draw[draw=black] (3,0) rectangle   (3.5,0.5);
\draw[draw=black] (3.5,0) rectangle   (4,0.5);
\draw[ draw=black] (4,0) rectangle   (4.5,0.5);

\draw[ draw=black] (0,-1) rectangle   (0.5,-0.5);
\draw[draw=black] (0.5,-1) rectangle   (1,-0.5);
\draw[ draw=black] (1,-1) rectangle   (1.5,-0.5);
\draw[ draw=black] (1.5,-1) rectangle   (2,-0.5);
\fill  (2.25, -0.75) circle[radius=1pt];
\fill  (2.50, -0.75) circle[radius=1pt];
\fill  (2.75, -0.75) circle[radius=1pt];
\draw[draw=black] (3,-1) rectangle   (3.5,-0.5);
\draw[draw=black] (3.5,-1) rectangle   (4,-0.5);

\node[rectangle] at (0.25,-1.25) {$1$};
\node[rectangle] at (0.75,-1.25) {$2$};
\node[rectangle] at (1.25,-1.25) {$3$};
\node[rectangle] at (1.75,-1.25) {$4$};

\node[rectangle] at (3.75,-1.25) {$n$};
\node[rectangle] at (4.5,-1.25) {$n+1$};


\end{tikzpicture}
\end{center}

For the left-hand side, place the $(n+1)$-board directly above the $n$-board as shown. Consider the location of the last common fault. Suppose the last fault is at the $k$th tile, where $1\le k \le n$. Then to avoid the fault after the $k$th tile, only dominos should be placed except for a single square placed on the tile $k+1$ in the row whose length is odd. The weight of the two boards for cells $1$ to $k$ is $f_{k+1} ^2$.  The dominos after the $k$th cell contribute  
$$(w^{f}_{k+1}w^{f}_{k+3} \dots w^{f}_{n-1} )(w^{f}_{k+2}w^{f}_{k+4} \dots w^{f}_{n}) = w^{f}_{k+1}w^{f}_{k+2} \dots w^{f}_{n-1} w^{f}_{n},$$ which is 
$$\frac{w^{f}_{1}w^{f}_{2} \dots w^{f}_{n}}{w^{f}_{1}w^{f}_{2} \dots w^{f}_{k} }.$$ 
Now we sum over $k$ to obtain the result. 
\end{proof}

\begin{Theorem} \label{Fib-61} Let $n$ be a positive integer. Then
\begin{equation*}
\sum_{k=0}^n g^n_{n-k}  f^{(2n-k)}_{k} 
=f_{2n} .
\end{equation*}
\end{Theorem}
\begin{proof}
We weigh tilings of an $(2n-1)$-board.

The right-hand side gives the total weight $f_{2n} $. 

For the left-hand side, consider the number $k$ of squares that appear among the first $n$ tiles. 
(There must be a square. If there were only dominos, then the board would have even length.) 
So the first $n$ tiles consist of $k$ squares and $(n-k)$ dominos, with weight $g^{n}_{n-k} $. 
The remaining board has length $k-1$. The weight is given by the shifted weighted Fibonacci number
$f^{(2n-k)}_{k} $.

Summing over $k$ gives the desired result.
\end{proof}

\begin{Theorem} \label{Fib-61b} Let $n$ be a non-negative integer. Then
\begin{equation*}
\sum_{k=0}^n g_{n-k}^n f_{k+1}^{(2n-k)}=f_{2n+1}.
\end{equation*}
\end{Theorem}
\begin{proof}
The proof is similar to that of Theorem~\ref{Fib-61} and is omitted.
\end{proof}   

\section{Fibonacci identities from telescoping}\label{sec:telescoping}

Many Fibonacci identities can be proved by telescoping. In this section we use the technique in \cite{GB2011, GB2016} to  obtain two more weighted Fibonacci identities. 

We use the alternate form of the telescoping lemma \cite[eq.\ (2.2)]{GB2011}, given by
\begin{equation}\label{telescoping-lemma2}
\sum_{k=1}^n t_k \frac{u_1u_2\cdots u_{k-1}}{v_1v_2\cdots v_{k}}
= \frac{u_1u_2\cdots u_{n}}{v_1v_2\cdots v_{n}} - 1,
\end{equation}
where $t_k=u_k-v_k.$

The form of Definition~\ref{def1} is such that \eqref{telescoping-lemma2} applies for the weighted Fibonacci numbers.  This was one motivation of using Definition~\ref{def1} rather than earlier definitions for the weighted Fibonacci numbers.

\begin{Theorem} Let $n$ be a positive integer. 
\begin{subequations}
\begin{gather}
\sum_{k=1}^n  (-1)^{k} \frac{f_{k+2} }{w^{f}_{1}w^{f}_{2} \dots w^{f}_{k}}
=(-1)^{n}\frac{f_{n+1} }{w^{f}_{1}w^{f}_{2} \dots w^{f}_{n}}-1. \label{Fib-5}
\\
\sum_{k=1}^n  w^{f}_{k-1}w^{f}_{k} \prod_{j=1}^k \frac{1}{1+w^{f}_{j}} f_{k-1} 
=1-\prod_{j=1}^n \frac{1}{1+w^{f}_{j}} f_{n+2} .\label{Fib-6}
\end{gather}
\end{subequations}
\end{Theorem}

\begin{Remark}
We may take special cases of \cite[Theorem~2.6]{GB2011} to obtain six identities. Four of them are already covered in \S\ref{sec:weigh}. It is convenient to simply use \eqref{telescoping-lemma2} rather than refer to previous results to obtain these. 
\end{Remark}

\begin{proof}
To prove \eqref{Fib-5}, take 
$$u_k = f_{k+1}   \;\; \text{ and }
v_k =-w^{f}_{k}f_{k}$$
in \eqref{telescoping-lemma2}.
Thus
${u_{k-1}}/{v_{k}} =-{1}/{w^{f}_{k}}$,
and 
$$t_k =u_k-v_k = f_{k+2}.$$

For \eqref{Fib-6}, take
$$
u_k = f_{k+2}  \;\;\text{ and }
v_k =(1+w^{f}_{k})f_{k+1} .
$$
Thus 
${u_{k-1}}/{v_{k}} ={1}/{(1+w^{f}_{k})}$, and
$$t_k =u_k-v_k = -w^{f}_{k}w^{f}_{k-1}f_{k-1}.$$
Substituting in \eqref{telescoping-lemma2}, we obtain the result.
\end{proof}

We conclude with another application of  \eqref{telescoping-lemma2},
which follows from Theorem~\ref{Fib-71}. 

\begin{Theorem}\label{Fib-11}
Let $r$ and $s$ be non-negative integers, and $n$ a positive integer. Then
\begin{multline}
\sum_{k=1}^n \frac{(-1)^k}{w_{r+1}^f w_{r+2}^f\cdots w_{r+k}^f}
\frac{f_{r+s+2k+1}}{f_{r+1}f_{s+1}^{(r)}}
\frac{f_{s+1}^{(r)}\cdots f_{s+k}^{(r+k-1)}}
{f_{s+1}^{(r+2)}\cdots f_{s+k}^{(r+k+1)}}
\\
=\frac{(-1)^n}{w_{r+1}^f w_{r+2}^f\cdots w_{r+n}^f}
\frac{f_{r+n+1}}{f_{r+1}}
\frac{f_{s+2}^{(r+1)}\cdots f_{s+n+1}^{(r+n)}}
{f_{s+1}^{(r+2)}\cdots f_{s+n}^{(r+n+1)}}
-1.
\end{multline}

\end{Theorem}
\begin{proof}
Take
$$
u_k=f_{r+k+1}f_{s+k+1}^{(r+k)}
\quad\text{and}\quad
v_k=-w_{r+k}^f f_{r+k}f_{s+k}^{(r+k+1)}.
$$
Thus
$$
\frac{u_{k-1}}{v_k}=
\frac{-f_{s+k}^{(r+k-1)}}{w_{r+k}^f f_{s+k}^{(r+k+1)}},
$$
and
$$
t_k=u_k-v_k=f_{r+s+2k+1}
$$
by virtue of Theorem~\ref{Fib-71}.
Substituting in \eqref{telescoping-lemma2} proves the result.
\end{proof}
We believe this identity is new even in the $q$-case. We remark that Theorem~\ref{Fib-81} yields yet another identity by telescoping. Finally, we mention that by varying the roles of $u_k$, $v_k$
and $t_k$ above we may get more identities.

\section*{Acknowledgements}
The research of Michael Schlosser was partially supported by the
Austrian Science Fund (FWF), grant~P~32305.
%

\begin{thebibliography}{10}

\bibitem{BQ2003}
A.~T. Benjamin and J.~J. Quinn.
\newblock {\em Proofs that really count}, volume~27 of {\em The Dolciani
  Mathematical Expositions}.
\newblock Mathematical Association of America, Washington, DC, 2003.
\newblock The art of combinatorial proof.

\bibitem{BQR2004}
A.~T. Benjamin, J.~J. Quinn, and J.~A. Rouse.
\newblock Fibinomial identities.
\newblock In {\em Applications of {F}ibonacci numbers. {V}ol. 9}, pages 19--24.
  Kluwer Acad. Publ., Dordrecht, 2004.

\bibitem{BCK2020a}
N.~Bergeron, C.~Ceballos, and J.~K\"{u}stner.
\newblock Elliptic and {$q$}-analogs of the {F}ibonomial numbers.
\newblock {\em S\'{e}m. Lothar. Combin.}, 84B:Art. 63, 12, 2020.

\bibitem{BCK2020b}
N.~Bergeron, C.~Ceballos, and J.~K\"{u}stner.
\newblock Elliptic and {$q$}-analogs of the {F}ibonomial numbers.
\newblock {\em SIGMA Symmetry Integrability Geom. Methods Appl.}, 16:Paper No.
  076, 16, 2020.

\bibitem{GB2011}
G.~Bhatnagar.
\newblock In praise of an elementary identity of {E}uler.
\newblock {\em Electron. J. Combin.}, 18(2):Paper 13, 44pp, 2011.

\bibitem{GB2016}
G.~Bhatnagar.
\newblock Analogues of a {F}ibonacci-{L}ucas identity.
\newblock {\em Fibonacci Quart.}, 54(2):166--171, 2016.

\bibitem{Cigler2003}
J.~Cigler.
\newblock Some algebraic aspects of {M}orse code sequences.
\newblock {\em Discrete Math. Theor. Comput. Sci.}, 6(1):55--68, 2003.

\bibitem{EDV1960}
D.~Everman, A.~E. Danese, and K.~Venkannayah.
\newblock Problem {E}1396.
\newblock {\em Amer. Math. Monthly}, 67:81--82, 1960.

\bibitem{Garrett2004}
K.~C. Garrett.
\newblock Weighted tilings and $q$-{F}ibonacci numbers, 2004.
\newblock Preprint.

\bibitem{GIS1999}
K.~C. Garrett, M.~E.~H. Ismail, and D.~Stanton.
\newblock Variants of the {R}ogers-{R}amanujan identities.
\newblock {\em Adv. in Appl. Math.}, 23(3):274--299, 1999.

\bibitem{GR90}
G.~Gasper and M.~Rahman.
\newblock {\em Basic {H}ypergeometric {S}eries}, volume~96 of {\em Encyclopedia
  of Mathematics and its Applications}.
\newblock Cambridge University Press, Cambridge, second edition, 2004.
\newblock With a foreword by Richard Askey.

\bibitem{MS2007}
M.~J. Schlosser.
\newblock Elliptic enumeration of nonintersecting lattice paths.
\newblock {\em J. Combin. Theory Ser. A}, 114(3):505--521, 2007.

\bibitem{MS2020b}
M.~J. Schlosser.
\newblock A noncommutative weight-dependent generalization of the binomial
  theorem.
\newblock {\em S\'{e}m. Lothar. Combin.}, 81:Art. B81j, 24, 2020.

\bibitem{SY2015}
M.~J. Schlosser and M.~Yoo.
\newblock Some combinatorial identities involving noncommuting variables.
\newblock In {\em Proceedings of {FPSAC} 2015}, Discrete Math. Theor. Comput.
  Sci. Proc., pages 961--972. Assoc. Discrete Math. Theor. Comput. Sci., Nancy,
  2015.

\bibitem{SY2016b}
M.~J. Schlosser and M.~Yoo.
\newblock An elliptic extension of the general product formula for augmented
  rook boards.
\newblock {\em European J. Combin.}, 58:247--266, 2016.

\bibitem{SY2016a}
M.~J. Schlosser and M.~Yoo.
\newblock Elliptic rook and file numbers.
\newblock In {\em 28th {I}nternational {C}onference on {F}ormal {P}ower
  {S}eries and {A}lgebraic {C}ombinatorics ({FPSAC} 2016)}, Discrete Math.
  Theor. Comput. Sci. Proc., BC, pages 1087--1098. Assoc. Discrete Math. Theor.
  Comput. Sci., Nancy, 2016.

\bibitem{SY2017b}
M.~J. Schlosser and M.~Yoo.
\newblock Elliptic extensions of the alpha-parameter model and the rook model
  for matchings.
\newblock {\em Adv. in Appl. Math.}, 84:8--33, 2017.

\bibitem{SY2017a}
M.~J. Schlosser and M.~Yoo.
\newblock Elliptic rook and file numbers.
\newblock {\em Electron. J. Combin.}, 24(1):Paper No. 1.31, 47, 2017.

\bibitem{SY2018}
M.~J. Schlosser and M.~Yoo.
\newblock Weight-dependent commutation relations and combinatorial identities.
\newblock {\em Discrete Math.}, 341(8):2308--2325, 2018.

\bibitem{SY2021}
M.~J. Schlosser and M.~Yoo.
\newblock Elliptic solutions of dynamical {L}ucas sequences.
\newblock {\em Entropy}, 23(2):Paper No. 183, 14, 2021.

\bibitem{Singh1985}
P.~Singh.
\newblock The so-called {F}ibonacci numbers in ancient and medieval {I}ndia.
\newblock {\em Historia Math.}, 12(3):229--244, 1985.

\bibitem{Vajda1989}
S.~Vajda.
\newblock {\em Fibonacci \& {L}ucas numbers, and the golden section}.
\newblock Ellis Horwood Series: Mathematics and its Applications. Ellis Horwood
  Ltd., Chichester; Halsted Press [John Wiley \& Sons, Inc.], New York, 1989.
\newblock Theory and applications, With chapter XII by B. W. Conolly.

\bibitem{WW1996}
E.~T. Whittaker and G.~N. Watson.
\newblock {\em A {C}ourse of {M}odern {A}nalysis}.
\newblock Cambridge Mathematical Library. Cambridge University Press,
  Cambridge, 1996.
\newblock An introduction to the general theory of infinite processes and of
  analytic functions; with an account of the principal transcendental
  functions, Reprint of the fourth (1927) edition.

\end{thebibliography}

\end{document}